\documentclass{aims}
\usepackage{amsmath}
  \usepackage{paralist}
  \usepackage{graphics,dsfont} 
  \usepackage{epsfig} 
 \usepackage[colorlinks=true]{hyperref}
\hypersetup{urlcolor=blue, citecolor=red}

\def\currentvolume{X}
 \def\currentissue{X}
  \def\currentyear{200X}
   \def\currentmonth{XX}
    \def\ppages{X--XX}

\newtheorem{theorem}{Theorem}[section]

\newtheorem{lemma}[theorem]{Lemma}

\theoremstyle{definition}
\newtheorem{definition}[theorem]{Definition}

\newcommand{\eps}{\varepsilon}

\newcommand{\R}{\mathbb{R}}
\newcommand{\Z}{\mathbb{Z}}

\newcommand{\seq}[1]{\left\{#1\right\}}

\newcommand{\Dx}{{\Delta x}}

\newcommand{\norm}[1]{\left\|#1\right\|}
\newcommand{\abs}[1]{\left|#1\right|}

\newcommand{\Dmx}{D_-}

\newcommand{\weakto}{\rightharpoonup}
\newcommand{\weakstarto}{\overset{\star}{\weakto}}
\newcommand{\CL}{\mathcal{L}}
\newcommand{\N}{\mathbb{N}}
\newcommand{\CMloc}{\mathcal{M}_{\mathrm{loc}}}
\newcommand{\Dm}{D_-}
\newcommand{\jint}{\int_{x_{j-1/2}}^{x_{j+1/2}}\!\!}
\newcommand{\loc}{{\mathrm{loc}}}

\newcommand{\Pt}{\Pi_T}

\title[Keyfitz-Kranzer]{Convergence of a finite difference scheme for 
  $2 \times 2$ Keyfitz-Kranzer system} 

\author[ Ujjwal Koley and Nils Henrik Risebro]{}


 \email{toujjwal@gmail.com}
 \email{nilshr@math.uio.no}

\begin{document}
\maketitle

\centerline{\scshape Ujjwal Koley}
\medskip
{\footnotesize
 \centerline{Institut f\"{u}r Mathematik,}
   \centerline{Julius-Maximilians-Universit\"{a}t W\"{u}rzburg,}
   \centerline{Campus Hubland Nord, Emil-Fischer-Straße 30,}
   \centerline{97074, W\"{u}rzburg, Germany.}
} 

\medskip

\centerline{\scshape Nils Henrik Risebro}
\medskip
{\footnotesize
 \centerline{ Centre of Mathematics for Applications (CMA)}
   \centerline{University of Oslo, P.O. Box 1053,}
   \centerline{Blindern, N--0316, Oslo, Norway}
}

\bigskip

\centerline{(Communicated by the associate editor name)}

\begin{abstract}
 We are concerned with the convergence of a numerical scheme for the
  initial value problem associated to the $2 \times 2$ Keyfitz-Kranzer system of
  equations. In this paper we prove the convergence of a finite difference scheme to a weak solution.
\end{abstract}

\section{Introduction}

In this paper, we are concerned with a symmetrically hyperbolic system
of two equations
\begin{equation}
  \label{eq:system}
  \begin{cases}
    u_t + (u \phi(r))_x = 0,  & \quad (x,t) \in \Pt \\
    v_t + (v \phi(r))_x = 0, & \quad (x,t) \in \Pt ,
  \end{cases}
\end{equation}
with initial data
\begin{equation}
  \label{eq:main_initial}
  \left( u(x,0), v(x,0) \right) = (u_0(x), v_0(x)), \quad x \in \R,
\end{equation}
where $r(x,t) = \sqrt{u^2(x,t) + v^2(x,t)}$, $\Pt=\R\times (0,T)$ with $T>0$ fixed, and $u,\, v:\Pt\to \R$ are the
unknown functions. Regarding the function $\phi$, the basic assumption is that $\phi: \R \rightarrow \R$ is a given ( sufficiently smooth ) scalar function (see Section ~\ref{sec:math} for precise assumptions). Systems of this type was first considered in \cite{KeyfitzKranzer} and later on by several other authors, as a prototypical example of a non-strict hyperbolic system. Note that \eqref{eq:system} is a non-strict hyperbolic system with first characteristic field is always linearly degenerate and the second characteristic field is either genuinely nonlinear or linearly
degenerate, depending on the behavior of $\phi$.

Due to the nonlinearity, discontinuities in the solution may appear independently of the smoothness of the initial data and weak solutions must be sought. 
A weak solution is defined as follows:
\begin{definition}
  We call the pair $(u,v)$ a weak solution of the Cauchy problem \eqref{eq:system}--\eqref{eq:main_initial} if 
\begin{itemize}
 \item [(a)] $u$ and $v$ are in $ L^{\infty}(\Pt)$.
  \item [(b)] $u$ and $v$ satisfy \eqref{eq:system} in the sense of distributions on $\Pt$, i.e., the following identities,
    \begin{equation}
\begin{aligned}
    \label{eq:weaksys_2}
    \iint_{\Pt} \!\! u \psi_t + u \phi(r) \psi_x \,dxdt + \int_\R u_0 \psi(x,0)\,dx & = 0, \\
    \iint_{\Pt} \!\! v \psi_t + v \phi(r) \psi_x \,dxdt + \int_\R v_0 \psi(x,0)\,dx & = 0,
\end{aligned}
  \end{equation}
hold for each smooth test function $\psi$ with compact support in $\Pt$.
  \end{itemize}
\end{definition}
In this paper, we propose a upwind semi discrete finite difference scheme and prove the convergence of the approximate solution to the weak solution of \eqref{eq:system}. In what follows, we first prove the strong convergence of approximate solution $r_\Dx = \sqrt{u_\Dx^2 + v_\Dx^2}$ using compensated compactness argument \cite{Chen,Lu}. Next, we prove a BV estimate of $ \varphi_\Dx := \tan^{-1}(\frac{u_\Dx}{v_\Dx})$. Then Helly's theorem combined with the strong convergence of $r_{\Dx}$ gives the required strong convergence of $u_{\Dx}$ and $v_{\Dx}$.

\section{Mathematical Framework}
\label{sec:math}
In this section we present some mathematical tools that we shall use in the analysis. To start with the basic assumptions on the initial data and the funtion $\phi(r)$, we assume that $\phi$ is a twice differentiable function $\phi:[0,\infty)\to [0,\infty)$ so that 
 \begin{itemize}
 \item [(a)]$\phi(r)>0$ and $ \phi'(r)\ge 0$ for all relevant $r$; 
 \item [(b)]$\phi(r), \phi'(r)$ and $\phi''(r)$ are bounded for all relevant $r$; 
 \item [(c)]$ \sqrt{u_0^2 + v_0^2} \in L^1(\R) \cap L^{\infty}(\R)$. 
 \end{itemize}
Note that we shall assume above assumptions throughout the paper. We also use the following compensated compactness result.
\begin{theorem}
  \label{thm:compcomp} Let $\Omega\subset \R\times \R^+$ be a bounded
  open set, and assume that $\seq{u^\eps}$ is a sequence of uniformly
  bounded functions such that $\abs{u^\eps}\le M$ for all $\eps$. Also
  assume that $f:[-M,M] \to \R$ is a twice differentiable
  function. Let $u^\eps \weakstarto u$ and $f(u^\eps)\weakstarto v$,
  and set
  \begin{equation}
    \begin{aligned}
      \left(\eta_1(s),q_1(s)\right) &= \left(s-k,f(s)-f(k)\right),\\
      \left(\eta_2(s),q_2(s)\right) &= \left(f(s)-f(k), \int_k^s
        (f'(\theta))^2\,d\theta \right),
    \end{aligned}\label{eq:entropies}
  \end{equation}
  where $k$ is an arbitrary constant. If
  \begin{equation*}
    \eta_i(u^\eps)_t + q_i(u^\eps)_x \ \text{ is in a compact set of
      $H^{-1}_{\mathrm{loc}}(\Omega)$ for $i=1$, $2$,}
  \end{equation*}
  then
  \begin{itemize}
  \item [(a)] $v=f(u)$, a.e.~$(x,t)$,
  \item [(b)] $u^\eps \to u$, a.e.~$(x,t)$ if $\mathrm{meas}\seq{u\,|\,
      f''(u)=0}=0$.
  \end{itemize}
\end{theorem}
For a proof of this theorem, see the monograph of Lu \cite{Lu}. The following compactness interpolation result (known as Murat's lemma \cite{Murat}) is useful in obtaining the $H^{-1}_{\loc}$ compactness needed in Theorem ~\ref{thm:compcomp}.
\begin{lemma}
  \label{lem:Murat} 
  Let $\Omega$ be a bounded open subset of $\R^2$.  Suppose
  that the sequence $\seq{\CL_\eps}_{\eps>0}$ of distributions is
  bounded in $W^{-1,\infty}(\Omega)$.  Suppose also that
  $$
  \CL_\eps=\CL_{1,\eps} + \CL_{2,\eps},
  $$
  where $\seq{\CL_{1,\eps}}_{\eps>0}$ is in a compact subset of
  $H^{-1}(\Omega)$ and $\seq{\CL_{2,\eps}}_{\eps>0}$ is in a bounded
  subset of $\CMloc(\Omega)$.  Then $\seq{\CL_\eps}_{\eps>0}$ is in a
  compact subset of $H^{-1}_{\mathrm{loc}}(\Omega)$.
\end{lemma}

\section{Semi Discrete Finite Difference Scheme}
\label{sec:semi}
We start by introducing the necessary notations. Given $\Dx>0$, we set $x_j=j\Dx$ and $x_{j\pm1/2} = x_j \pm \Dx/2$ for $j\in \Z$ and for any function $u = u(x)$, we define $u_j = u(x_j)$. Let $D_{\pm}$ denote the discrete forward and backward differences, i.e.,
\begin{equation*}
  D_{\pm}u_j = \mp\frac{u_j - u_{j\pm 1}}{\Dx}. 
\end{equation*}
To a sequence $\seq{w_j}_{j\in \Z}$ we associate the function $w_\Dx$ defined by 
\begin{equation*}
  w_{\Dx}(x)= \sum_{j \in \Z} w_j \mathds{1}_{I_j}(x),
\end{equation*}
where $\mathds{1}_{A}$ denotes the characteristic function of the set $A$. 
We will use following standard notations:
\begin{equation*}
\begin{aligned}
 & \norm{u_{\Dx}}_{L^{\infty}(\R)} = \sup_{j \in \Z} \abs{u_j}, \quad  \norm{u_\Dx}_{L^1(\R)} = \Dx \sum_{j\in\Z} \abs{u_j}, \\
 & \norm{u_\Dx}_{L^2(\R)} = \left( \Dx \sum_{j\in\Z} \abs{u_j}^2 \right)^{1/2}, \quad \abs{u_{\Dx}}_{BV(\R)} = \sum_{j \in \Z} \abs{u_j - u_{j-1}}.
\end{aligned}
\end{equation*}

Now let $\seq{u_j(t)}_{j\in\Z}$ and $\seq{v_j(t)}_{j\in\Z}$ satisfy
the (infinite) system of ordinary differential equations,
\begin{equation}    \label{eq:discrete}
  \begin{cases}
    u_j' + \Dmx \left(\phi(r_j) u_j\right)=0,\\ 
    v_j' + \Dmx \left(\phi(r_j) v_j\right)=0,
  \end{cases}
\end{equation}
with initial values
\begin{equation}\label{eq:discrete_init}
  u_j(0)=\frac{1}{\Dx}\int_{x_{j-1/2}}^{x_{j+1/2}}\! u_0(x)\,dx\
  \text{ and }\ 
  v_j(0)=\frac{1}{\Dx}\int_{x_{j-1/2}}^{x_{j+1/2}}\! v_0(x)\,dx.
\end{equation}
Here $r_j = \sqrt{u_j^2 + v_j^2}$. It is natural to view \eqref{eq:discrete}
as an ordinary differential equation in $L^2(\R)\times L^2(\R)$. It is easy to show that the right hand side of \eqref{eq:discrete}
is Lipschitz continuous in $L^2(\R)\times L^2(\R)$, which essentially gives the local (in time) existence and uniqueness of differentiable
solutions. 
The next lemma shows that the $L^2$ norm remains bounded if it
is bounded initially, so  the solution can be defined up to any
time.
\begin{lemma}
  \label{lem:rl1}
  Let $\seq{u_j(t)}$, $\seq{v_j(t)}$ be defined by
  \eqref{eq:discrete}, and let $r_j = \sqrt{u_j^2+v_j^2}$. Then
  \begin{equation*}
\begin{aligned}
    \norm{r_\Dx(t)}_{L^1(\R)} & \le \norm{r_\Dx(0)}_{L^1(\R)}, \\
    \norm{r_\Dx(t)}_{L^2(\R)} & \le \norm{r_\Dx(0)}_{L^2(\R)}.
\end{aligned}
  \end{equation*}
  Furthermore, there is a constant $C$, independent of $\Dx$ and $T$,
  such that 
  \begin{equation*}
    \int_0^T \left( \sum_j \int_{r_{j-1}}^{r_j}\!\!\! \left(r_j^2-s^2\right)\phi'(s)\,ds  
    + \Dx \sum_{j} \phi_{j-1}
    \Dx \left(\left(D_+u_j\right)^2+\left(D_+v_j\right)^2\right) \right)\,dt \le C.  
  \end{equation*}
\end{lemma}

\begin{proof}
Set $U=(u,v)$ and observe that $r_j = \abs{U_j}$. We can rewrite the system \eqref{eq:discrete} as
\begin{equation}
 \label{eq:rewrite}
\begin{aligned}
 (U_j)_t + \Dmx (U_j \, \phi(r_j)) =0.
\end{aligned}
\end{equation}
Let $\eta = \eta(U)$ be a differentiable function $\eta:\R^2\to \R$, take
the inner product of \eqref{eq:rewrite} with $\nabla_U\eta(U_j)$ to
get 
\begin{multline}
  \label{eq:entropy1}
  \frac{d}{dt} \eta(U_j) + \Dm \left(\phi_j \eta(U_j) \right) \\
  + 
  \left[\left(\nabla_U\eta(U_j),U_j\right)-\eta(U_j)\right]\Dm\phi_j + 
  \phi_{j-1} \frac{\Dx}{2} d^2_U \eta_{j-1/2}\left(\Dm U_j,\Dm U_j\right)=0.
\end{multline}
Here $\phi_j=\phi(r_j)$, and $d^2\eta$ denotes the Hessian matrix of
$\eta$, so that 
\begin{equation*}
  d^2_U\eta_{j-1/2} = d^2_U\eta(U_{j-1/2})
\end{equation*}
for some $U_{j-1/2}$ between $U_j$ and $U_{j-1}$. By a limiting
argument, the function $\eta(U)=\abs{U}$ can be used. This function is
convex, i.e., $d^2_U\abs{U} \ge 0$. This means that
\begin{equation}
  \label{eq:rbnd}
  \frac{d}{dt} r_j + \Dm(r_j\phi(r_j) \le 0.
\end{equation}
Multiplying by $\Dx$ and summing over $j$ we get
\begin{equation}
  \label{eq:l1rbnd}
  \norm{r_\Dx(t)}_{L^1(\R)} \le \norm{\abs{U_0}}_{L^1(\R)}.
\end{equation}
Furthermore, choosing $\eta(U)=\abs{U}^2$ in \eqref{eq:entropy1} we get 
\begin{equation*}
  \frac{d}{dt} r_j^2(t) + \Dm \left(r_j^2\phi_j\right) + r_j^2 \Dm
  \phi_j + \phi_{j-1}\Dx\abs{\Dm U_j}^2 = 0.
\end{equation*}
We have that
\begin{align*}
  \Dm\left(r_j^2\phi_j\right) + r_j^2\Dm\phi_j &=
  \frac{2}{\Dx}\int_{r_{j-1}}^{r_j}\!\!\! s\phi(s) + s^2\phi'(s)\,ds + 
  \frac{1}{\Dx}\int_{r_{j-1}}^{r_j}\!\!\! \left(r_j^2-s^2\right)
  \phi'(s)\,ds\\
  &=\Dm g(r_j) +   \frac{1}{\Dx}\int_{r_{j-1}}^{r_j}\!\!\! \left(r_j^2-s^2\right)
  \phi'(s)\,ds, 
\end{align*}
where 
\begin{equation}\label{eq:gdef}
  g(r)=2\int_0^r s\phi(s)+s^2\phi'(s)\,ds.
\end{equation}
Using this we find that
\begin{equation}
  \label{eq:l2bnd}
  \norm{r_\Dx(t)}_{L^2(\R)} \le \norm{\abs{U_0}}_{L^2(\R)},
\end{equation}
since, by the assumption that $\phi'\ge 0$, 
\begin{equation*}
  \int_{r_{j-1}}^{r_j}\!\!\! \left(r_j^2-s^2\right)\phi'(s)\,ds \ge 0.
\end{equation*}
Hence $\norm{\left(u_\Dx(t),v_\Dx(t)\right)}_{L^2(\R)^2}$ is bounded independently of $\Dx$
and $t$. Therefore, the exists a differentiable solution $(u_\Dx(t), v_\Dx(t))$ to
\eqref{eq:discrete} for all $t>0$. Furthermore, we have the bound 
\begin{equation*}
 \int_0^T \left(\sum_j \int_{r_{j-1}}^{r_j}\!\!\! \left(r_j^2-s^2\right)\phi'(s)\,ds  
  + \Dx\sum_j\phi_{j-1} \Dx \abs{\Dm U_j}^2 \right) \,dt\le C,
\end{equation*}
for some constant $C$ which is independent of $t$ and $\Dx$.
\end{proof}

\begin{lemma}
  \label{lem:linfty} 
  If there is a constant $R$ such that $r_j(0)\le R$ for all $j$, then
  $r_j(t)\le R$ for all $j$ and $t>0$. 

  If $0<u_j(0)$ and $0<v_j(0)$ for all
  $j$, and  there is a constant $C>0$ such that 
  \begin{equation*}
    \frac{1}{C} \le \frac{u_j(0)}{v_j(0)}\le C,
  \end{equation*}
  then 
  \begin{equation*}
    \frac{1}{C} \le \frac{u_j(t)}{v_j(t)}\le C
  \end{equation*}
  for all $j$ and $t>0$.
\end{lemma}
\begin{proof}
  If $j_0$ is such that $r_{j_0}(t_0)\ge r_{j_0-1}(t_0)$,
  then $D_-f(r_{j_0}(t_0))\ge 0$ with $f(r)=r \phi(r)$, since $f$ is non-decreasing. Hence,
  from \eqref{eq:rbnd}, we see that $r_{j_0}'(t_0)\le 0$. This proves
  the first statement of the lemma. 

  To prove the second statement, we first show that if $u_j(0)>0$,
  then $u_j(t)\ge 0$, and if $u_{j_0}(t_0)=0$ for some $t_0>0$ and
  $j_0$, then $u_j(t)=0$ for all $j\le j_0$ and all $t\ge t_0$. A
  similar statement holds for $v_j$. To see this, note that
  \begin{equation*}
    u_j' + u_j D_-\phi_j + \phi_{j-1}D_-u_j = 0.
  \end{equation*}
  Assume that for some $t_0$ and $j_0$, $u_{j_0}(t_0)=0$ and
  $u_{j_0}(t)\ge 0$ for
  $t<t_0$. Then $u_{j_0}'(t_0)\le 0$. If $u_{j_0-1}(t_0)>0$, this leads to a
  contradiction, hence $u_{j_0-1}(t_0)=0$. By repeating the argument
  we get that  that $u_j(t_0)=0$ for all $j<j_0$. If both $u_{j}$ and
  $u_{j-1}$ are zero, then $u_j'(t)=0$, hence if $u_j(t_0)=0$,
  $u_j(t)=0$ for all $t>t_0$. A similar statement holds for
  $v_j$. This means that if $r_{j_0}(t_0)=0$, both $u_{j_0}(t_0)$ and
  $v_{j_0}(t_0)$ are  zero, hence $r_{j}(t)=0$ for $j\le j_0$ and
  $t\ge t_0$.

  Let for the moment $\varphi_j$ be defined by 
  \begin{equation}\label{eq:angledef}
    \varphi_j = \begin{cases} \tan^{-1}\left(\frac{v_j}{u_j}\right)
      &\text{if $r_j>0$,}\\
      \varphi_{j+1} &\text{if $r_j=0$}.
    \end{cases}
  \end{equation}
  By the previous observation, we know that $0\le \varphi_j \le
  \pi/2$. Now if $r_j>0$,
  \begin{align*}
    \varphi_j'(t) &= \frac{1}{1+(v_j/u_j)^2}
    \left(\frac{v_j}{u_j}\right)' =-\frac{u_ju_{j-1}\phi_{j-1}}{r_j^2} D_-\left(\frac{v_j}{u_j}\right).
  \end{align*}
  Therefore $\varphi_j$ satisfies the equation
  \begin{equation}
    \label{eq:angle}
    \varphi_j' + \frac{u_j u_{j-1}\phi_{j-1}}{r_j^2}
    D_-\left(\tan\left(\varphi_j\right)\right) = 0.
  \end{equation}
  This equation holds for any $j$ where $r_j>0$, if $r_{j_0}=0$ for
  some $j_0$, then we define $\varphi_{j}(t)=\varphi_{j_0+1}(t)$ for
  all $j\le j_0$. 
  
  We have that  $\tan$ is an increasing function, and
  $(u_ju_{j-1}\phi_{j-1})/r_j^2 \ge 0$ if $r_j>0$. Therefore, if
  $\varphi_j(t)>\varphi_{j-1}(t)$, then $\varphi_j'(t)\le 0$. Similarly if
  $\varphi_j(t)< \varphi_{j-1}(t)$, then $\varphi_j'(t)\ge 0$. The
  assumption on the initial data implies that
  \begin{equation*}
    0<\inf_{j}\varphi_j(0) \le \varphi_j(t) \le \sup_j \varphi_j(0)<\pi/2.
  \end{equation*}
  Incidentally, this shows that if $u_j(t)=0$, then $v_j(t)=0$ and
  vice versa. 
\end{proof}

Let now $\eta_i(r)$ and $q_i(r)$ be given by \eqref{eq:entropies} for $i=1$,
$2$. We then have that 
\begin{equation}
  \label{eq:eta1}
  \frac{d}{dt}\eta_1(r_j) + \Dm (q_1(r_j)) + e_{1,j} = 0,
\end{equation}
where 
\begin{align*}
  f(r)&=r\phi(r),\ \ \   q_1(r)=f(r)-f(k)\ \ \text{and } \\ 
  e_{1,j}&= \phi_{j-1} \Dx \left(\Dm U_j\right)^T \frac{1}{r_{j-1/2}}
  \left(I-\frac{U_{j-1/2}\otimes
      U_{j-1/2}}{r_{j-1/2}^2}\right)\left(\Dm U_j\right).
\end{align*}
We shall now find an equation satisfied by $\eta_2$. It is easy to see that
\begin{equation*}
  \frac{d}{dt} f(r_j) + q_2'(r_j) \Dm r_j 
  - f'(r_j)\frac{\Dx}{2}
  f''\left(r_{j-1/2}\right) \left(\Dm r_j\right)^2 + f'(r_j)e_{1,j}=0.
\end{equation*}
Set 
\begin{equation*}
  e_{2,j}=\frac{\Dx}{2}
  f''\left(r_{j-1/2}\right) \left(\Dm r_j\right)^2.
\end{equation*}
This can be rewritten as 
\begin{equation}
 \label{eq:eta2}
\begin{aligned}
  \frac{d}{dt}\eta_2(r_j) + \Dm q_2(r_j) + \frac{\Dx}{2}
  q_2''(r_{j-1/2}) \left(\Dm r_j\right)^2 - f'(r_j) \left(e_{2,j} -
    e_{1,j}\right) = 0.
\end{aligned}
\end{equation}
Finally set 
\begin{equation*}
  e_{3,j} =  \frac{\Dx}{2}
  q_2''(r_{j-1/2}) \left(\Dm r_j\right)^2,
\end{equation*}
and 
\begin{equation*}
  e_{i}(x,t)=e_{i,j}(t)\ \text{ for $x\in (x_{j-1/2},x_{j+1/2}]$ and $i=1,2,3$.}
\end{equation*}
\begin{lemma}\label{lem:ei}
  We have that $e_i\in \CMloc(\Pt)$ for $i=1$, $2$, $3$. 
\end{lemma}
\begin{proof}
  The result follows since $\phi(r)>0$ and
  \begin{equation*}
    \int_0^T \Dx\sum_j \Dx\abs{\Dm r_j}^2 \le C,
  \end{equation*}
\end{proof}

\begin{lemma}
 \label{lem:compact}
Let $(u_{\Dx}, v_{\Dx})$ be generated by the scheme \eqref{eq:discrete} and let $r_{\Dx}$ be defined by $r_{\Dx} = \sqrt{u_{\Dx}^2 + v_{\Dx}^2}$. Then 
\begin{align*}
\seq{\eta_i(r_\Dx)_t + q_i(r_\Dx)}_{\Dx>0} \,\, \text{is compact in} \,\, H^{-1}_{\loc}(\Pt),
\end{align*}
where $\eta_i$ and $q_i$ are given by \eqref{eq:entropies}. 
\end{lemma}

\begin{proof}
Let $i=1$ or $i=2$, and $\psi$ is a test function in $H^1_{\mathrm{loc}}(\Pt)$. we define
\begin{align*}
\langle \mathcal{L}_{i}, \psi \rangle & = \langle \eta_i(r_\Dx)_t + q_i(r_{\Dx})_x , \psi\rangle \\
  &=  \int_0^T \sum_j \jint \frac{d}{dt}\eta_i(r_j) \psi(x,t) + 
  \Dm q_i(r_j) \psi(x_{j-1/2},t)\,dx\, dt\\
  &=\int_0^T \sum_j \jint \left(\frac{d}{dt}\eta_i(r_j) + \Dm q_i(r_j)
  \right)\psi(x,t)\,dx\,dt \\
  &\qquad + \int_0^T \sum_j \jint
  \left(\psi(x_{j-1/2},t)-\psi(x,t)\right)\Dm q_i(r_j) \,dx\,dt\\
  &=: \langle \mathcal{L}_{i,1},\psi\rangle + \langle
  \mathcal{L}_{2,i},\psi\rangle.
\end{align*}
By \eqref{eq:eta1}, \eqref{eq:eta2} and Lemma~\ref{lem:ei} we know
that $\mathcal{L}_{i,1}\in \CMloc(\Omega)$. Regarding
$\mathcal{L}_{i,2}$ we have
\begin{align*}
  \abs{\langle
    \mathcal{L}_{2,i},\psi\rangle} &= \Bigl|
  \int_0^T \sum_j \jint \int_{x_{j-1/2}}^x \psi_x(y,t)\,dy\, \Dm
  q_i(r_j)\,dx\,dt \Bigr|\\
  &\le \int_0^T \sum_j \Dx^{3/2} \Bigl(
  \jint  \left(\psi_x(x,t)\right)^2\,dx\Bigr)^{1/2}
  \norm{q_i'}_{L^\infty} \abs{\Dm r_j} \,dt \\
  &\le C \sqrt{\Dx} \norm{\psi}_{H^1(\Pt)}.
\end{align*}
Therefore the above estimate shows that $\mathcal{L}_{2,i}$ is compact in
$H^{-1}(\Pt)$. By Lemma~\ref{lem:Murat}, we conclude the sequence 
$\seq{\eta_i(r_\Dx)_t + q_i(r_\Dx)}_{\Dx>0}$ is compact in
$H^{-1}_{\mathrm{loc}}(\Pt)$. 
\end{proof}

\begin{lemma}
  \label{lem:rconverg} If 
  \begin{equation*}
    \mathrm{meas}\seq{ r \,\Bigm|\, 2\phi'(r) + r\phi''(r) = 0} = 0,
  \end{equation*}
  then there is a subsequence of $\seq{\Dx}$ and a function $r$
  such that $r_{\Dx} \to r$ a.e.~$(x,t)\in \Pt$. We have that $r\in
  L^{\infty}([0,T];L^1(\R))$.
\end{lemma}
\begin{proof}
  The strong convergence of $r_\Dx$ follows from the compensated
  compactness theorem, Theorem~\ref{thm:compcomp} and the compactness
  of $\seq{\eta_i(r_\Dx)_t + q_i(r_\Dx)_x}_{\Dx>0}$ for $i=1,2$. 
\end{proof}
\begin{lemma}
  \label{lem:anglconv}
  If $r_j(0)>0$ and there is a positive constant $C$ such that $1/C
  \le (v_j(0)/u_j(0)) \le C$, and
  \begin{equation*}
    \abs{\frac{v_0}{u_0}}_{B.V.}<\infty,
  \end{equation*}
  then there is a subsequence of $\seq{\Dx}$ and a function
  $\varphi\in C([0,T];L^1_{\mathrm{loc}}(\R))$ such that
  $\varphi_{\Dx}(\cdot,t)\to \varphi(\cdot,t)$ in
  $L^1_{\mathrm{loc}}(\R)$ as $\Dx \to 0$.
\end{lemma}
\begin{proof}
  For $j$ such that $r_j>0$, by \eqref{eq:angle}
  \begin{equation*}
    \frac{d}{dt}\varphi_j +
    \frac{u_ju_{j-1}\phi_{j-1}}{r_j^2 \cos^2\left(\varphi_{j-1/2}\right)} 
    \Dm\varphi_j = 0,
  \end{equation*}
  where $\varphi_{j-1/2}$ is some intermediate value. Set 
  \begin{equation*}
    \mu_j =
    \frac{u_ju_{j-1}\phi_{j-1}}{r_j^2 \cos^2\left(\varphi_{j-1/2}\right)} 
    \text{ and }\ \theta_j = \Dm \varphi_j.
  \end{equation*}
  Note that $\mu_j\ge 0$, and that $\mu_j$ is bounded since $\varphi_j<\pi/2$.
  Then $\theta_j$ satisfies
  \begin{equation}\label{eq:thetadef}
    \frac{d}{dt}\theta_j + \mu_{j-1}\Dm\theta_j + \theta_j \Dm \mu_{j-1}=0.
  \end{equation}
  Let $\eta_\alpha(\theta)$ be a smooth approximation to
  $\abs{\theta}$ such that 
  \begin{equation*}
    \eta''_\alpha(\theta)\ge 0 \ \text{ and }\ \lim_{\alpha\to
      0}\eta_\alpha(\theta) = \lim_{\alpha\to
      0}\left(\theta\eta'_\alpha(\theta)\right)=\abs{\theta}.
  \end{equation*}
  We multiply \eqref{eq:thetadef} by $\eta'_\alpha(\theta_j)$ to get
  an equation satisfied by $\eta_\alpha(\theta_j)$. Observe that 
  \begin{align*}
    \mu_{j-1}\eta_\alpha'(\theta_j)\Dm\theta_j & +
    \theta_j\eta_\alpha'(\theta_j)\Dm \mu_j \\
&= \mu_{j-1}\Dm\eta_\alpha(\theta_j) +\theta_j\eta'_\alpha(\theta_j)
    \Dm \mu_j + \frac{\Dx}{2}\mu_{j-1}
    \eta''_\alpha(\theta_{j-1/2})\left(\Dm\theta_j\right)^2 \\
    &\ge 
    \Dm\left(\mu_j\eta_\alpha(\theta_j)\right) +
    \left(\theta_j\eta'_\alpha(\theta_j)-\eta_\alpha(\theta_j)\right)\Dm \mu_j.
  \end{align*}
  Hence
  \begin{equation*}
    \frac{d}{dt}\eta_\alpha(\theta_j) +
    \Dm\left(\mu_j\eta_\alpha(\theta_j)\right) \le
    \left(\eta_\alpha(\theta_j)-\theta_j\eta'_\alpha(\theta_j)\right)\Dm \mu_j.
  \end{equation*}
  Now let $\alpha\to 0$ to obtain
  \begin{equation}
    \label{eq:abstheta}
    \frac{d}{dt}\abs{\theta_j} + \Dm\left(\mu_j\abs{\theta_j}\right)
    \le 0.
  \end{equation}
  If we multiply this with $\Dx$, sum over $j$ and integrate in $t$,
  we find that 
  \begin{equation*}
    \abs{\varphi_\Dx(\cdot,t)}_{B.V} \le
    \abs{\varphi_{\Dx}(\cdot,0)}_{B.V.} \le \abs{\varphi(\cdot,0)}_{B.V.}<\infty.
  \end{equation*}
  By Helly's theorem, for each $t\in [0,T]$, the sequence
  $\seq{\varphi_{\Dx}(\cdot,t)}_{\Dx>0}$ has a subsequence which
  converges strongly in $L^1_{\mathrm{loc}}(\R)$. By using a diagonal
  argument, we can get this convergence for a dense countable set
  $\seq{t_k}_{k\in \N}\subset [0,T]$. Since $\varphi_{\Dx}(\cdot,t)$
  has bounded variation, it is $L^1_{\mathrm{loc}}$ Lipschitz
  continuous in $t$, that is
  \begin{equation*}
    \norm{\varphi_{\Dx}(\cdot,t)-\varphi_\Dx(\cdot,s)}_{L^1_{\mathrm{loc}}(\R)}
    \le \sup_j\mu_j 
    \abs{\varphi(\cdot,0)}_{B.V.} \abs{t-s}.
  \end{equation*}
  This means that $\varphi_\Dx(\cdot,t)$ converges also for $[0,T]\ni
  t\not\in\seq{t_k}_{k\in \N}$. Furthermore, it also shows that
  $\varphi=\lim_{\Dx\to 0}\varphi_{\Dx}$ is continuous in $t$, with
  values in $L^1_{\mathrm{loc}}(\R)$. 
\end{proof}
Now we have the strong convergence of $r_\Dx$ and of
$\varphi_\Dx$. This means that also $u_\Dx$ and $v_\Dx$ converge
strongly to some functions $u$ and $v$ in $L^{\infty}([0,T];L^1_{\mathrm{loc}}(\R))$ since we have 
\begin{equation}
  \label{eq:trigonometry}
  u_{\Dx}=r_\Dx \cos(\varphi_\Dx) \ \text{ and }\ 
  v_{\Dx}=r_\Dx \sin(\varphi_\Dx).
\end{equation}
\begin{theorem}
  \label{thm:convergence} Let $\phi$ be a twice differentiable
  function $\phi:[0,\infty)\to [0,\infty)$ such that $\phi(r)>0$ and
  $\phi'(r)\ge 0$, and 
  \begin{equation*}
    \mathrm{meas}\seq{ r\,\bigm|\, 2\phi'(r)+ r\phi''(r)=0} = 0.
  \end{equation*}
  Let $u_\Dx$ and $v_\Dx$ be defined by
  \eqref{eq:discrete}--\eqref{eq:discrete_init}. If $u_0>0$, $v_0>0$  and
  $u_0^2+v_0^2\in L^1(\R)$, and $\abs{v_0/u_0}_{B.V.}<\infty$ then
  there exists functions $u$ and $v$ in
  $L^{\infty}([0,T];L^1_{\mathrm{loc}}(\R))$ such that $u_\Dx\to u$ and
  $v_\Dx\to v$ as $\Dx\to 0$. The functions $u$ and $v$ are weak
  solutions to \eqref{eq:system}.
\end{theorem}
\begin{proof}
  We have already established convergence. It remains to show that $u$ and $v$ are weak solutions. To this end,
  observe that\footnote{Here we ``extend'' the definition of $D_-$ and
    $D_+$ to arbitrary functions in the obvious manner.}
  \begin{equation*}
    \int_0^T \int_\R \Dm\left(u_\Dx\phi(r_\Dx)\right) \psi(x,t)\,dxdt
    = -  \int_0^T \int_\R u_\Dx\phi(r_\Dx) D_+\psi(x,t)\,dxdt.
  \end{equation*}
  As $\Dx\to 0$, $D_+\psi\to \psi_x$ for any $\psi \in
  C^1_0(\Omega)$. This means that $u$ is a weak solution. Similarly we can show that $v$ is a weak solution. Hence, the functions $u$ and $v$ are weak
  solutions to \eqref{eq:system}.
\end{proof}

\section*{Acknowledgments}

The first author acknowledges support from the {\bf Alexander von Humboldt Foundation}, through a Humboldt Research Fellowship for postdoctoral researchers.

\medskip
Received xxxx 20xx; revised xxxx 20xx.
\medskip

\end{document}